\documentclass[12pt]{article}

\usepackage{amsthm,amsmath,amssymb}
\usepackage{hyperref}

\usepackage{geometry}
\renewcommand{\ge}{\geqslant}
\renewcommand{\le}{\leqslant}

\newcommand{\arxiv}[1]{\href{http://arxiv.org/abs/#1}{\texttt{arXiv:#1}}}

\theoremstyle{plain}
\newtheorem{theorem}{Theorem}
\newtheorem{lemma}[theorem]{Lemma}

\newtheorem*{main}{Main Theorem}

\newcommand{\abs}[1]{\left\lvert #1 \right\rvert}
\DeclareMathOperator*{\E}{\mathbb{E}}

\title{Tomaszewski's problem on\\randomly signed sums:\\breaking the 3/8 barrier}

\author{Ravi B. Boppana\\
\small Department of Mathematics\\[-0.8ex]
\small M. I. T.\\[-0.8ex]
\small Massachusetts, USA\\
\small\tt rboppana@mit.edu
\and
Ron Holzman\\
\small Department of Mathematics\\[-0.8ex]
\small Technion--Israel Institute of Technology\\[-0.8ex]
\small Israel\\
\small\tt holzman@tx.technion.ac.il
}

\date{\small August 31, 2017}

\begin{document}

\maketitle

\begin{abstract}
Let $v_1$, $v_2$, $\ldots\,$, $v_n$ be real numbers whose squares add up to~$1$.
Consider the $2^n$ signed sums of the form $S = \sum \pm v_i$.
Holzman and Kleitman (1992) proved that at least $\frac{3}{8}$ of these sums satisfy $\lvert S \rvert \le 1$.
This $\frac{3}{8}$ bound seems to be the best their method can achieve.
Using a different method, we improve the bound to $\frac{13}{32}$,
thus breaking the $\frac{3}{8}$ barrier.

\end{abstract}

\section{Introduction}

Let $v_1$, $v_2$, \dots, $v_n$ be real numbers such that the sum of their squares is at most~$1$.
Consider the $2^n$ signed sums of the form $S = \pm v_1 \pm v_2 \pm \dots \pm v_n$.
In 1986,
B.~Tomaszewski (see Guy~\cite{Guy}) asked the following question:
is it always true that at least $\frac{1}{2}$ of these sums
satisfy $\abs{S} \le 1$?
Most examples with $n = 2$ and $v_1^2 + v_2^2 = 1$ show that $\frac{1}{2}$ can't be replaced with a bigger number.

Holzman and Kleitman~\cite{HK} proved that at least $\frac{3}{8}$ of the sums satisfy $\abs{S} \le 1$.
This result was an immediate consequence of their main result:
at least $\frac{3}{8}$ of the sums satisfy the strict inequality $\abs{S} < 1$,
provided that each $\abs{v_i}$ is strictly less than~1.
This $\frac{3}{8}$ bound for $\abs{S} < 1$ is best possible:
consider the example with $n=4$ and $v_1 = v_2 = v_3 = v_4 = \frac{1}{2}$.
So $\frac{3}{8}$ seems to be a natural barrier to their method of proof.

Using a different method,
we prove that more than $\frac{13}{32}$ of the sums satisfy $\abs{S} \le 1$.
In other words, we break the $\frac{3}{8}$ barrier.
Our method, roughly speaking, goes like this.
We will let the first few $\pm$ signs be arbitrary.
But once the partial sum becomes near~1 in absolute value,
we will show that the final sum still has a decent chance of remaining at most~1 in absolute value.

We can actually improve the $\frac{13}{32}$ bound a tiny bit, to $\frac{13}{32} + 9 \times 10^{-6}$.  
Combining our method with other ideas, which could handle the tight cases for our analysis, may lead to further improvements of the bound. Still, the conjectured lower bound of $\frac{1}{2}$ currently appears to be out of reach.

Ten years after Holzman and Kleitman~\cite{HK} but independently,
Ben-Tal, Nemirovski, and Roos~\cite{BNR} proved that at least $\frac{1}{3}$ of the sums satisfy $\abs{S} \le 1$;
they say that the proof is mainly due to P.~van der Wal.
Shnurnikov~\cite{Shnurnikov} refined the argument of~\cite{BNR} to
prove a $36\%$ bound.
Even though these two bounds are weaker than that of Holzman and Kleitman,
the methods used to prove them are noteworthy.
In particular,
we will use the conditioning argument of~\cite{BNR}
and the fourth moment method of~\cite{Shnurnikov}.

Let Tomaszewski's constant be the largest constant~$c$ such that the fraction of sums that satisfy $\abs{S} \le 1$ is always at least~$c$.
We now know that Tomaszewski's constant is between $\frac{13}{32}$ and $\frac{1}{2}$.
Both \cite{HK} and \cite{BNR} conjecture that Tomaszewski's constant is $\frac{1}{2}$.
De, Diakonikolas, and Servedio~\cite{DDS} developed an algorithm to approximate Tomaszewski's constant.
Specifically, given an $\epsilon > 0$,
their algorithm will output a number that is within $\epsilon$ of Tomaszewski's constant.
The running time of their algorithm is exponential in $1/\epsilon^3$,
so it's not clear that we can run their algorithm in a reasonable amount of time to improve the known bounds
on Tomaszewski's constant.

The conjectured lower bound of $\frac{1}{2}$ has been confirmed in some special cases.
For example, von Heymann~\cite{Heymann} and Hendriks and van Zuijlen~\cite{HvZ} proved the conjecture when $n \le 9$.
Also, van Zuijlen~\cite{Zuijlen} and von Heymann~\cite{Heymann} proved the conjecture when all of the $\abs{v_i}$ are equal.  

We will use the language of probability.
Let $\Pr[A]$ be the probability of an event~$A$.
Let $\E(X)$ be the expected value of a random variable~$X$.
A \emph{random sign} is a random variable whose probability distribution is the uniform distribution on the set~$\{ -1, +1 \}$.
With this language, we can restate our main result.

\begin{main}
Let $v_1$, $v_2$, \dots, $v_n$ be real numbers such that $\sum_{i=1}^n v_i^2$ is at most~$1$.
Let $a_1$, $a_2$, \dots, $a_n$ be independent random signs.
Let $S$ be $\sum_{i = 1}^n a_i v_i$.
Then
$\Pr[ \abs{S} \le 1 ] > \frac{13}{32}.$
\end{main}

In Section~\ref{sec:initial_bound} of this paper,
we will provide a short proof of a bound better than~$\frac{3}{8}$.
In Section~\ref{sec:best_bound}, we will refine the analysis to improve the bound to~$\frac{13}{32}$ and slightly beyond.

\section{Beating the 3/8 bound}
\label{sec:initial_bound}

In this section,
we will give the simplest proof we can of a bound better than~$\frac{3}{8}$.
Namely, we will prove a bound of $\frac{37}{98}$,
which is a little more than $37.75\%$.
In Section~\ref{sec:best_bound},
we will improve the bound further.

We begin with a lemma.
Roughly speaking,
this lemma can be used to show that if a partial sum is a little less than~1,
then the final sum has a decent chance of remaining less than~1 in absolute value.

\begin{lemma} \label{lem:initial_bound}
Let $x$ be a real number such that $\abs{x} \le 1$.
Let $v_1$, $v_2$, \dots, $v_n$ be real numbers such that
\[
  \sum_{i = 1}^n v_i^2 \le \frac{2}{7} (1 + \abs{x})^2 .
\]
Let $a_1$, $a_2$, \dots, $a_n$ be independent random signs.
Let $Y$ be $\sum_{i=1}^n a_i v_i$.
Then
\[
  \Pr[ \abs{x + Y} \le 1 ] \ge \frac{37}{98} \, .
\]
\end{lemma}

\begin{proof}
By symmetry, we may assume that $x \ge 0$.
The fourth moment of~$Y$ is
\[
  \E(Y^4)
	  = 3 \Bigl( \sum_{i = 1}^n v_i^2 \Bigr)^2 - 2 \sum_{i = 1}^n v_i^4
	  \le 3 \Bigl( \sum_{i = 1}^n v_i^2 \Bigr)^2
		\le \frac{12}{49} (1 + x)^4 .
\]
So, by the fourth moment version of Chebyshev's inequality\footnote{Shnurnikov~\cite{Shnurnikov}
used the fourth moment in a similar situation.},
\[
  \Pr[\abs{Y} \ge 1 + x]
	  \le \frac{\E(Y^4)}{(1 + x)^4}
		\le \frac{12}{49} \, .
\]
Looking at the complement,
\[
  \Pr[\abs{Y} < 1 + x] \ge \frac{37}{49} \, .
\]
Because $Y$ has a symmetric distribution,
\[
  \Pr[-1 - x < Y \le 0]
	  \ge \frac{1}{2} \Pr[\abs{Y} < 1 + x]
		\ge \frac{37}{98} .
\]
Recall that $x \le 1$.
Hence if $-1 - x < Y \le 0$, then $\abs{x + Y} \le 1$.
Therefore
\[
  \Pr[ \abs{x + Y} \le 1 ]
		\ge \Pr[-1 - x < Y \le 0]
		\ge \frac{37}{98} \, .
\]
\end{proof}

Next we will use Lemma~\ref{lem:initial_bound} to go beyond the $\frac{3}{8}$ bound.

\begin{theorem} \label{thm:initial_bound}
Let $v_1$, $v_2$, \dots, $v_n$ be real numbers such that $\sum_{i=1}^n v_i^2$ is at most~$1$.
Let $a_1$, $a_2$, \dots, $a_n$ be independent random signs.
Let $S$ be $\sum_{i = 1}^n a_i v_i$.
Then
\[
  \Pr[ \abs{S} \le 1 ] \ge \frac{37}{98} \, .
\]
\end{theorem}

\begin{proof}
By inserting $0$'s,
we may assume that $n \ge 4$.
By permuting,
we may assume that the four largest $\abs{v_i}$ are
$\abs{v_n} \ge \abs{v_1} \ge \abs{v_{n-1}} \ge \abs{v_2}$.
By the quadratic mean inequality,
\[
  \frac{\abs{v_1} + \abs{v_2} + \abs{v_{n-1}} + \abs{v_n}}{4}
	  \le \sqrt{ \frac{v_1^2 + v_2^2 + v_{n-1}^2 + v_n^2}{4} }
		\le \sqrt{ \frac{1}{4} }
		= \frac{1}{2} \, .
\]
So $\abs{v_1} + \abs{v_2} + \abs{v_{n-1}} + \abs{v_n} \le 2$.
Because of our ordering,
\[
  \abs{v_1} + \abs{v_2}
	  \le \frac{\abs{v_1} + \abs{v_n}}{2} + \frac{\abs{v_2} + \abs{v_{n-1}}}{2}
		\le 1.
\]

Given an integer~$t$ from $0$ to $n$,
let $X_t$ be the partial sum~$\sum_{i=1}^t a_i v_i$
and let $Y_t$ be the remaining sum~$\sum_{i = t + 1}^n a_i v_i$.
Let $T$ be the smallest nonnegative integer~$t$ such that
$t = n - 1$ or $\abs{X_t} > 1 - \abs{v_{t + 1}}$.
In a stochastic process such as ours,
$T$ is called a stopping time, defined by the stopping rule in the previous sentence\footnote{A similar stopping rule
was implicitly used by Ben-Tal \emph{et al.}~\cite{BNR} and refined by Shnurnikov~\cite{Shnurnikov}.
In addition, \cite{Shnurnikov} pointed out the value of having $\abs{v_1} + \abs{v_2} \le 1$.}.
Note that $T \ge 2$, since $\abs{v_1} + \abs{v_2} \le 1$.
By the stopping rule, $\abs{X_{T - 1}} \le 1 - \abs{v_T}$.
Hence by the triangle inequality,
\[
  \abs{X_T}
		\le \abs{X_{T-1}} + \abs{v_T}
		\le 1 - \abs{v_T} + \abs{v_T}
		= 1.
\]
Also by the stopping rule, if $T < n - 1$, then $\abs{X_T} > 1 - \abs{v_{T + 1}}$.

We will condition on $T$ and $X_T$.
We claim that
\[
  \Pr[\abs{S} \le 1 \mid T, X_T ] \ge \frac{37}{98} \, .
\]
By averaging over $T$ and $X_T$,
this claim implies the theorem.
To prove the claim, we may assume by symmetry that $X_T \ge 0$.
We will divide the proof of the claim into three cases, depending on~$T$.

\medskip

\textbf{Case 1:}
$T = n - 1$.
In this case, $\abs{Y_T} = \abs{v_n} \le 1$.
Recall that $0 \le X_T \le 1$.
Hence if $Y_T \le 0$, then $\abs{S} = \abs{X_T + Y_T} \le 1$.
Therefore by symmetry,
\[
  \Pr[\abs{S} \le 1 \mid T, X_T ]
		\ge \Pr[ Y_T \le 0 \mid T, X_T ]
		\ge \frac{1}{2} \, .
\]

\medskip

\textbf{Case 2:}
$T = n - 2$.
In this case,
\[
  \abs{Y_T}
	  \le \abs{v_{n-1}} + \abs{v_n}
		\le 2 - \abs{v_1}
		\le 2 - \abs{v_{n - 1}}.
\]
Recall that $1 - \abs{v_{n-1}} < X_T \le 1$.
Hence if $Y_T \le 0$,
then $\abs{S} = \abs{X_T + Y_T} \le 1$.
Therefore by symmetry,
\[
  \Pr[\abs{S} \le 1 \mid T, X_T ]
		\ge \Pr[ Y_T \le 0 \mid T, X_T ]
		\ge \frac{1}{2} \, .
\]

\medskip

\textbf{Case 3:}
$T \le n - 3$.
In this case, by the stopping rule,
\[
  \sum_{i = T + 1}^n v_i^2
	  \le 1 - \sum_{i = 1}^T v_i^2
		\le 1 - v_1^2 - v_2^2
		\le 1 - 2 v_{T + 1}^2
		< 1 - 2(1 - X_T)^2 .
\]
We can bound the final expression as follows:
\[
  1 - 2(1 - X_T)^2
	  = \frac{2}{7} (1 + X_T)^2 - \frac{1}{7} (4 X_T - 3)^2
		\le \frac{2}{7} (1 + X_T)^2 .
\]
Hence the hypotheses of Lemma~\ref{lem:initial_bound} are satisfied with $x = X_T$ and $Y = Y_T$.
By Lemma~\ref{lem:initial_bound}, we conclude that
\[
  \Pr[\abs{S} \le 1 \mid T, X_T ]
	  = \Pr[ \abs{X_T + Y_T} \le 1 \mid T, X_T ]
		\ge \frac{37}{98} \, .
\]
\end{proof}

\section{Further improvement}
\label{sec:best_bound}

In this section, we will improve the lower bound to~$\frac{13}{32}$, which is $40.625\%$.
At the end, we will sketch how to improve the bound further, to $\frac{13}{32} + 9 \times 10^{-6}$.

Let us examine where the proof of Theorem~\ref{thm:initial_bound} is potentially tight.
Looking at its Case~3, we see that the proof is potentially tight when $T = 2$
and $\abs{v_1} = \abs{v_2} = \abs{v_3} = \frac{1}{4}$.
But that scenario is impossible: if $T = 2$, then by the stopping rule, $\abs{v_1} + \abs{v_2} > 1 - \abs{v_3}$.
This suggests that we can sharpen the bound on $\sum_{i=T+1}^n v_i^2$ in terms of $T$ and $X_T$.

Another idea is that our final bound on $\Pr[\abs{S} \le 1]$, instead of being the worst-case conditional bound, may be taken to be a weighted average of the conditional bounds, with weights corresponding to the distribution of $T$.

First, we state the following generalization of Lemma~\ref{lem:initial_bound}.
Given a number~$c$, define $F(c)$ by
\[
  F(c) = \frac{1}{2} ( 1 - 3 c^2 ) .
\]

\begin{lemma} \label{lem:best_bound}
Let $c$ be a nonnegative number.
Let $x$ be a real number such that $\abs{x} \le 1$.
Let $v_1$, $v_2$, \dots, $v_n$ be real numbers such that
\[
  \sum_{i = 1}^n v_i^2 \le c (1 + \abs{x})^2 .
\]
Let $a_1$, $a_2$, \dots, $a_n$ be independent random signs.
Let $Y$ be $\sum_{i=1}^n a_i v_i$.
Then
\[
  \Pr[ \abs{x + Y} \le 1 ] \ge F(c) .
\]
\end{lemma}

\begin{proof}
By symmetry, we may assume that $x \ge 0$.
As in the proof of Lemma~\ref{lem:initial_bound}, the fourth moment of~$Y$ satisfies
\[
  \E(Y^4)
	  \le 3 \Bigl( \sum_{i=1}^n v_i^2 \Bigr)^{2}
		\le 3 c^{2} (1 + x)^4 .
\]
So, by the fourth moment version of Chebyshev's inequality,
\[
  \Pr[\abs{Y} \ge 1 + x]
	  \le \frac{\E(Y^4)}{(1 + x)^4}
		\le  3 c^{2} .
\]
Following the proof of Lemma~\ref{lem:initial_bound}, by taking the complement and then using the symmetry of~$Y$,
we have
\[
  \Pr[ \abs{x + Y} \le 1 ] \ge \frac{1}{2} ( 1 - 3 c^{2} ) = F(c) .
\]
\end{proof}

Now we will use Lemma~\ref{lem:best_bound} to prove our $\frac{13}{32}$ lower bound.  

\begin{theorem}  \label{thm:best_bound}
Let $v_1$, $v_2$, \dots, $v_n$ be real numbers such that $\sum_{i=1}^n v_i^2$ is at most~$1$.
Let $a_1$, $a_2$, \dots, $a_n$ be independent random signs.
Let $S$ be $\sum_{i = 1}^n a_i v_i$.
Then
\[
  \Pr[ \abs{S} \le 1 ] > \frac{13}{32}.
\]
\end{theorem}

\begin{proof}
By inserting $0$'s,
we may assume that $n \ge 4$.
By symmetry, we may assume that each $v_i$ is nonnegative.
By permuting,
we may assume that the $v_i$ are ordered as follows:
\[
  v_n \ge v_1 \ge v_{n-1} \ge v_2 \ge v_3 \ge \dots \ge v_{n-2} .
\]
Except for the oddballs $v_n$ and $v_{n-1}$, the order is decreasing.
As in Theorem~\ref{thm:initial_bound},
we have $v_1 + v_2 + v_{n-1} + v_n \le 2$
and $v_1 + v_2 \le 1$.

Given an integer~$t$ from $0$ to $n$,
let $M_t$ be the sum~$\sum_{i=1}^t v_i$.
Let $K$ be the smallest nonnegative integer~$t$ such that
$t = n - 1$ or $M_t > 1 - v_{t + 1}$.
The parameter~$K$ measures how spread out the $v_i$ are.
Note that $K \ge 2$, since $v_1 + v_2 \le 1$.
By the definition of~$K$, observe that $M_{K - 1} \le 1 - v_K$ and hence $M_K \le 1$.
Also, if $K < n - 1$, then $M_K > 1 - v_{K + 1}$ and hence $M_{K + 1} > 1$.

Given an integer~$t$ from $0$ to $n$,
define the sums $X_t$ and $Y_t$ as in Theorem~\ref{thm:initial_bound}.
Note that $\abs{X_t} \le M_t$.
Following Theorem~\ref{thm:initial_bound}, let $T$ be the smallest nonnegative integer~$t$ such that
$t = n - 1$ or $\abs{X_t} > 1 - v_{t + 1}$.
Note that $T \ge K$.
As before, we have $\abs{X_{T - 1}} \le 1 - v_T$ and $\abs{X_T} \le 1$.
Also, if $T < n - 1$, then $\abs{X_T} > 1 - v_{T + 1}$.

We will bound from below the conditional probability $\Pr[\abs{S} \le 1 \mid T ]$.
Namely, we will prove the two-piece lower bound
\[
  \Pr[\abs{S} \le 1 \mid T ] \ge
		\begin{cases}
		  F\left( \frac{(K + 1)^2 - T}{(2K + 1)^2} \right)  &  \textrm{if } T \le \frac{3K+2}{2} \, ; \\
			F\left( \frac{K}{4K + 2} \right)                  &  \textrm{if } T \ge \frac{3K+2}{2} \, .
		\end{cases}
\]
We will actually prove the same lower bound on the refined conditional probability $\Pr[\abs{S} \le 1 \mid T, X_T ]$.
To prove this claim,
we may assume by symmetry that $X_T \ge 0$.
We will divide the proof of the claim into five cases, depending on~$T$.

\medskip

\textbf{Case 1:} $T = n - 1$.
The proof of this case is the same as Case~1 of Theorem~\ref{thm:initial_bound},
which yields the stronger bound $\Pr[\abs{S} \le 1 \mid T, X_T ] \ge \frac{1}{2}$.

\medskip

\textbf{Case 2:} $T = n - 2$.
The proof of this case is the same as Case~2 of Theorem~\ref{thm:initial_bound},
which yields the stronger bound $\Pr[\abs{S} \le 1 \mid T, X_T ] \ge \frac{1}{2}$.

\medskip

\textbf{Case 3:} $K + 1 \le T \le \frac{3K + 2}{2}$ and $T \le n - 3$.
By the quadratic mean inequality,
\[
  \sum_{i = 1}^{K + 1} v_i^2
	  \ge \frac{1}{K + 1} \Bigl( \sum_{i = 1}^{K + 1} v_i \Bigr )^2
		= \frac{1}{K + 1} M_{K + 1}^2
		> \frac{1}{K + 1} \, .
\]
Hence, by splitting our sum into two parts, we get
\[
  \sum_{i = 1}^T v_i^2
		> \frac{1}{K + 1} + (T - K - 1) v_{T + 1}^2
		\ge \frac{1}{K + 1} + (T - K - 1) (1 - X_T)^2 .
\]
As a simpler bound,
\[
  \sum_{i = 1}^T v_i^2
	  \ge T v_{T + 1}^2
		> T (1 - X_T)^2 .
\]
Multiplying the second-to-last inequality by $\frac{2T - K - 1}{2K + 1}$ and the last inequality by $\frac{3K + 2 - 2T}{2K + 1}$,
both multipliers being nonnegative by the case assumption, we get
\[
  \sum_{i = 1}^T v_i^2
	  \ge \frac{2T - K - 1}{(K + 1)(2K + 1)} + \frac{(K + 1)^2 - T}{2K + 1} (1 - X_T)^2 .
\]
Therefore, looking at the complementary sum, we get
\[
  \sum_{i = T+1}^n v_i^2
	  \le \frac{(K + 1)^2 - T}{(K + 1)(2K + 1)} \left[ 2 - (K + 1)(1 - X_T)^2 \right] .
\]
We can bound the bracketed expression as follows:
\begin{align*}
  2 - (K + 1)(1 - X_T)^2
	  &= \frac{K+ 1}{2K + 1} (1 + X_T)^2 - \frac{2}{2K+ 1} \left[ (K + 1) X_T - K \right]^2 \\
		&\le \frac{K + 1}{2K + 1} (1 + X_T)^2 .
\end{align*}
Plugging this inequality back into the previous one, we get
\[
  \sum_{i = T+1}^n v_i^2
	  \le \frac{(K + 1)^2 - T}{(2K + 1)^2} (1 + X_T)^2 .
\]
Hence the hypotheses of Lemma~\ref{lem:best_bound} are satisfied with $c = \frac{(K + 1)^2 - T}{(2K + 1)^2}$,
$x = X_T$, and $Y = Y_T$.
By Lemma~\ref{lem:best_bound}, we conclude that
\[
  \Pr[\abs{S} \le 1 \mid T, X_T ]
		= \Pr[ \abs{X_T + Y_T} \le 1 \mid T, X_T ]
		\ge F\Bigl( \frac{(K + 1)^2 - T}{(2K + 1)^2} \Bigr) .
\]

\medskip

\textbf{Case 4:} $\frac{3K + 2}{2} \le T \le n - 3$.
As in Case~3,
we can bound $\sum_{i = 1}^T v_i^2$ as follows:
\[
  \sum_{i = 1}^T v_i^2
		> \frac{1}{K + 1} + (T - K - 1) (1 - X_T)^2 .
\]
Because $T \ge \frac{3K + 2}{2}$, this inequality implies
\[
  \sum_{i = 1}^T v_i^2
		\ge \frac{1}{K + 1} + \frac{K}{2} (1 - X_T)^2 .
\]
Compare this bound with the combined bound from Case~3:
\[
  \sum_{i = 1}^T v_i^2
	  \ge \frac{2T - K - 1}{(K + 1)(2K + 1)} + \frac{(K + 1)^2 - T}{2K + 1} (1 - X_T)^2 .
\]
Note that our bound on $\sum_{i=1}^T v_i^2$ is the same as this bound from Case~3 when $T = \frac{3K + 2}{2}$.
So we can repeat the remainder of Case~3 to get the same lower bound on $\Pr[\abs{S} \le 1 \mid T, X_T ]$ when $T = \frac{3K + 2}{2}$.
The bound on $\Pr[\abs{S} \le 1 \mid T, X_T ]$ in Case~3 was
\[
  \Pr[\abs{S} \le 1 \mid T, X_T ]
		\ge F\Bigl( \frac{(K + 1)^2 - T}{(2K + 1)^2} \Bigr).
\]
When $T = \frac{3K + 2}{2}$, this bound becomes
\[
  \Pr[\abs{S} \le 1 \mid T, X_T ]
		\ge F \Bigl( \frac{K}{4K + 2} \Bigr) .
\]		
So we get the same bound in our current case.

\medskip

\textbf{Case 5:} $T = K \le n - 3$.
By the quadratic mean inequality,
\[
  \sum_{i = 1}^{T} v_i^2
	  \ge \frac{1}{T} \Bigl( \sum_{i = 1}^{T} v_i \Bigr )^2
		= \frac{1}{T} M_{T}^2
		\ge \frac{1}{T} X_{T}^2
		= \frac{1}{K} X_T^2 .
\]
We can bound the final expression as follows:
\begin{align*}
  \frac{1}{K} X_T^2
	  &= \frac{1}{K + 1} - (1 - X_T)^2 + \frac{1}{K(K + 1)} \left[ (K + 1) X_T - K \right]^2 \\
		&\ge \frac{1}{K + 1} - (1 - X_T)^2 .
\end{align*}
Plugging this inequality back into the previous one, we get
\[
  \sum_{i = 1}^{T} v_i^2
	  \ge \frac{1}{K + 1} - (1 - X_T)^2
		= \frac{1}{K + 1} + (T - K - 1) (1 - X_T)^2 .
\]
This is the same inequality we derived at the beginning of Case~3.
So we can repeat the remainder of Case~3 to get the same lower bound:
\[
  \Pr[\abs{S} \le 1 \mid T, X_T ]
    \ge  F\Bigl( \frac{(K + 1)^2 - T}{(2K + 1)^2} \Bigr) .
\]

\medskip

In summary, we have proved our claim on conditional probability:
\[
  \Pr[\abs{S} \le 1 \mid T ] \ge
		\begin{cases}
		  F\left( \frac{(K + 1)^2 - T}{(2K + 1)^2} \right)  &  \textrm{if } T \le \frac{3K+2}{2} \, ; \\
			F\left( \frac{K}{4K + 2} \right)                  &  \textrm{if } T \ge \frac{3K+2}{2} \, .
		\end{cases}
\]
Next, we will use this conditional bound to derive a lower bound on the unconditional probability $\Pr[\abs{S} \le 1]$.

As mentioned above, we always have $T \ge K$.
In fact, assuming that $K \le n-4$, we have $T=K$ if the signs $a_1$, \dots, $a_K$ are all equal, and otherwise $T \ge K+2$.
This follows from observing that if $a_1$, \dots, $a_K$ are not all equal,
then $\abs{X_K} \le \sum_{i=1}^{K-1} v_i - v_K \le 1 - v_{K+1}$ and $\abs{X_{K+1}} \le \sum_{i=1}^{K-1} v_i - v_K + v_{K+1} \le 1 - v_{K+2}$, by the definition of $K$ and the ordering of the $v_i$.

This shows that for $K \le n-4$ we have $\Pr[T=K]=\frac{1}{2^{K-1}}$ and $\Pr[T \ge K+2]=1-\frac{1}{2^{K-1}}$. Therefore
\begin{align*}
\Pr[&\abs{S} \le 1] \\
  &= \frac{1}{2^{K-1}} \Pr[\abs{S} \le 1 \mid T=K] + \Bigl( 1 - \frac{1}{2^{K-1}} \Bigr) \Pr[\abs{S} \le 1 \mid T \ge K+2] \\
  &\ge \frac{1}{2^{K-1}} F\left( \frac{(K + 1)^2 - K}{(2K + 1)^2} \right)
	   + \Bigl(1 - \frac{1}{2^{K-1}} \Bigr) F\left( \frac{(K + 1)^2 - (K+2)}{(2K + 1)^2} \right).
\end{align*}
Here we have used our conditional bounds, the fact that they are nondecreasing in $T$, and the inequality $K+2 \le \frac{3K+2}{2}$. Note that this lower bound on $\Pr[\abs{S} \le 1]$ remains valid without assuming that $K \le n-4$. Indeed, if $K=n-3$ it is still true that $\Pr[T=K] = \frac{1}{2^{K-1}}$, and while $T=K+1=n-2$ may occur in this case, it yields a conditional bound of $\frac{1}{2}$ as shown in Case~2 above, which is even better than our stated lower bound. The values $n-2$ and $n-1$ for $K$ are of course covered by the conditional bound of $\frac{1}{2}$ in Cases 1 and 2 above.

Thus, to conclude our proof it suffices to show that
\[
  \frac{1}{2^{K-1}} F\left( \frac{(K + 1)^2 - K}{(2K + 1)^2} \right)
	  + \Bigl(1 - \frac{1}{2^{K-1}} \Bigr) F\left( \frac{(K + 1)^2 - (K+2)}{(2K + 1)^2} \right)
		  > \frac{13}{32}
\]
holds for all $K \ge 2$. Substituting the relevant expressions into the formula for $F$ and performing routine manipulations, the latter is shown to be equivalent to
\[ 64(K^2 + K) < 2^{K-1} (40K^2 + 40K - 15), \]
which indeed holds for $K \ge 2$.
\end{proof}

Can we improve this $\frac{13}{32}$ lower bound?
Yes, a little.  
The idea is to replace the fourth moment with the more flexible $p$th moment,
where $p$ is a parameter to be optimized.
To do so, we will need Khintchine's inequality.
This inequality was first proved by Khintchine~\cite{Khintchine} in a weaker form
and later proved by Haagerup~\cite{Haagerup} with the optimal constants.
Namely, given $p \ge 2$,
let $B_p$ be the constant
\[
  B_p = \frac{2^{p/2} \Gamma(\frac{p + 1}{2})}{\sqrt{\pi}} \, ,
\]
where $\Gamma$ is the gamma function.
For example, $B_2 = 1$, $B_3 = 2 \sqrt{2/\pi}$, and $B_4 = 3$.

\begin{theorem}[Khintchine's inequality] \label{thm:khintchine}
Let $p$ be a real number such that $p \ge 2$.
Let $v_1$, $v_2$, \dots, $v_n$ be real numbers.
Let $a_1$, $a_2$, \dots, $a_n$ be independent random signs.
Let $S$ be $\sum_{i = 1}^n a_i v_i$.
Then
\[
  \E(\abs{S}^p) \le B_p \Bigl( \sum_{i=1}^n v_i^2 \Bigr)^{p/2} .
\]
\end{theorem}

For the improved lower bound, choose with foresight $p = 3.95937$.
In Lemma~\ref{lem:best_bound}, replace the fourth moment with the $p$th moment 
and apply Khintchine's inequality (with $S = Y$),
which allows us to replace the function~$F$ with the function~$G$ defined by 
$G(c) = \frac{1}{2} ( 1 - B_p c^{p/2} )$.
Use this revised lemma in Theorem~\ref{thm:best_bound}.     
The resulting lower bound is~$G(\frac{1}{4})$, which is bigger than $\frac{13}{32} + 9 \times 10^{-6}$.  
We omit the details.

\end{document}